\documentclass{article}

\usepackage[a4paper, total={6in, 9in}]{geometry}
\usepackage[utf8]{inputenc}
\usepackage{graphicx}
\usepackage{amsmath,amsfonts,amsthm,mathrsfs,amssymb,cite}
\usepackage{tikz}
\usepackage{cases}
\usepackage{color}

\allowdisplaybreaks

\newtheorem{thm}{Theorem}[section]

\newtheorem{cor}[thm]{Corollary}
\newtheorem{lem}{Lemma}[section]
\newtheorem{prop}{Proposition}[section]
\theoremstyle{definition}

\theoremstyle{remark}

\newtheorem{rem}{Remark}[section]
\numberwithin{equation}{section}
\definecolor{rot}{rgb}{0.000,0.000,0.000}

\newcommand{\R}{\mathbb{R}}

\newcommand{\N}{\mathbb{N}}

\newcommand{\be}{\begin{eqnarray}}
\newcommand{\ben}{\begin{eqnarray*}}
\newcommand{\en}{\end{eqnarray}}
\newcommand{\enn}{\end{eqnarray*}}

\newcommand{\supp}{\mathop{\rm supp}}

\newcommand{\nrm}[2][]{ \| {#2} \|_{#1}} 

\makeatletter
\newcommand*{\Relbarfill@}{\arrowfill@\Relbar\Relbar\Relbar}
\newcommand*{\xeq}[2][]{\ext@arrow 0055\Relbarfill@{#1}{#2}}
\makeatother

\newcommand{\rot}{\textcolor[rgb]{0,0,0} }

\usepackage{lineno} 

\title{Inverse source problems \rot{in an inhomogeneous medium} with a single far-field pattern}
\author{Guanghui Hu\footnotemark[1]\;\; and\; Jingzhi Li\footnotemark[2]}
\date{}
\begin{document}
\maketitle
\renewcommand{\thefootnote}{\fnsymbol{footnote}}
\footnotetext[1]{Beijing Computational Science Research Center (CSRC), Beijing 100193, P. R. China (hu@csrc.ac.cn).}
\footnotetext[2]{Department of Mathematics, Southern University of Science and Technology (SUSTech), Shenzhen 518005, P. R. China (li.jz@sustech.edu.cn).}
\renewcommand{\thefootnote}{\arabic{footnote}}

	\begin{abstract}
	This paper concerns time-harmonic inverse source problems with a single far-field pattern in two dimensions, where the  source term is compactly supported in an \emph{a priori} given inhomogeneous background medium.
  For convex-polygonal source terms, we prove that the source
support together with the zeroth and first order derivatives of the source function at corner
points can be uniquely determined. Further, we prove that an admissible set of source functions (including harmonic functions) having a convex-polygonal support can be uniquely identified by a single far-field pattern. A class of radiating sources is characterized and \rot{the extension} of the radiated field across a corner point is proven impossible. The corner scattering theory leads to a data-driven inversion scheme for imaging an arbitrarily convex-polygonal source support.
\vspace*{0.5cm}

\end{abstract}

\section{Introduction and main results}
Consider the radiating of a time-harmonic acoustic source  in an inhomogeneous background medium.  In two dimensions, this can be modeled by the inhomogeneous Helmholtz equation
	\begin{subequations} \label{eq:1-cornscatter2018}
		\begin{numcases}{}
		\displaystyle{ \Delta u + k^2 n(x) u = f \quad \mbox{in} \quad \R^2, }\medskip \label{eq:1a-cornscatter2018} \\
		\displaystyle{ \lim\limits_{r \to +\infty} \sqrt{r} \left( \frac {\partial u} {\partial r} - ik u \right) = 0, \quad r := |x|. }\medskip \label{eq:1b-cornscatter2018}
		\end{numcases}
	\end{subequations}	
In this paper, the refractive index function is supposed to satisfy $n(x)\in C^{0,\alpha'}(\R^2)$ ($0<\alpha'<1$) and
	$n(x)=1$ in  $|x|>R$ for some $R>0$. The number $k>0$ represents the wavenumber of the homogeneous medium in $|x|>R$ and $f\in L^2(\R^2)$ is a source term compactly supported in  $D\subset B_R:=\{x: |x|<R\}$ (see Figure \ref{fig1}).
\begin{figure}[htbp]
  \centering
  \includegraphics[width = 3in]{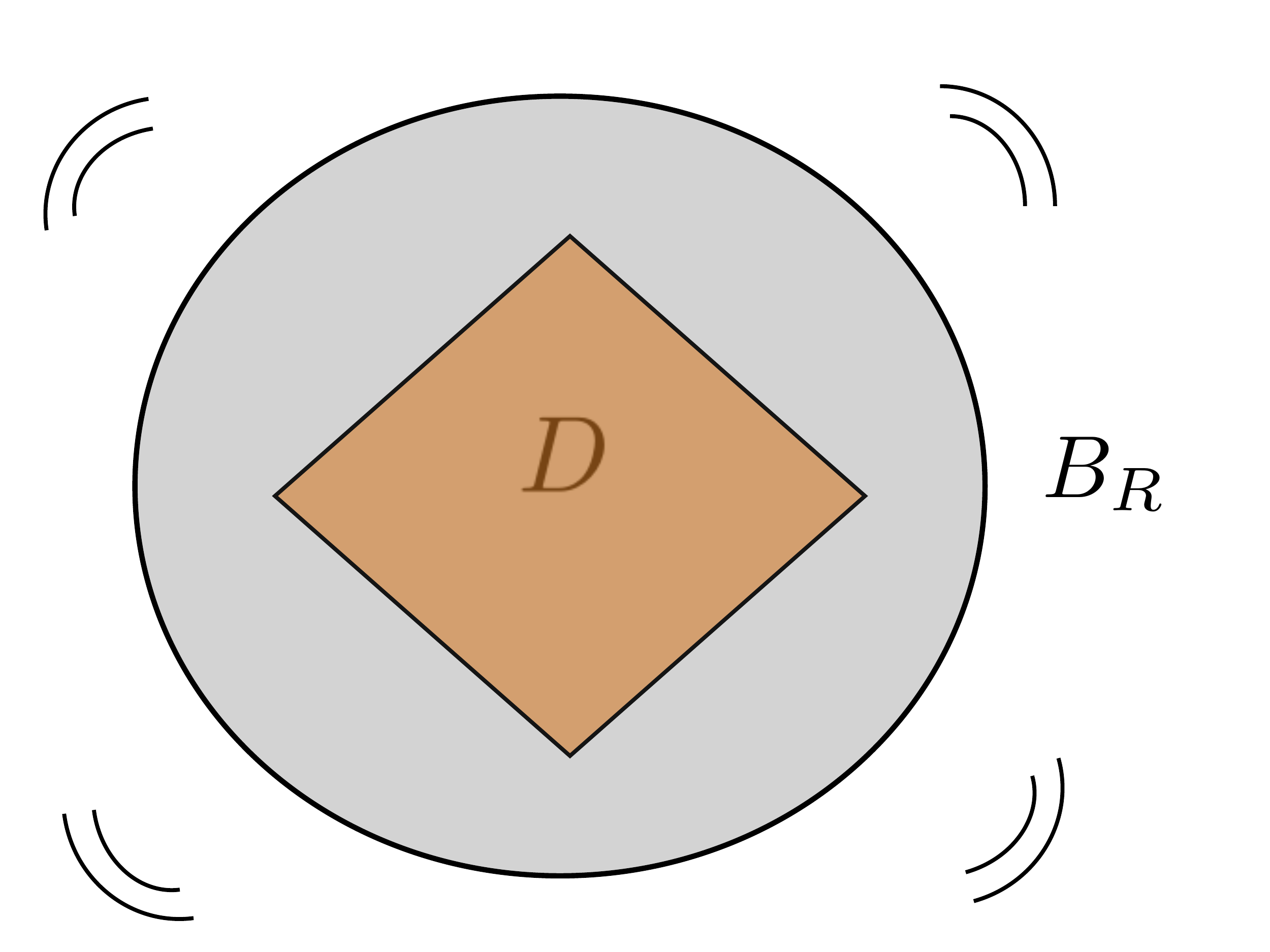}
   \caption{Illustration of the acoustic source problem in an inhomogeneous background medium $B_R$. We assume that the source support  $D\subset B_R$ is a convex polygon and  that the medium in the exterior of $B_R$ is homogeneous.}
      \label{fig1}
    \end{figure}
The  Sommerfeld Radiation Condition \eqref{eq:1b-cornscatter2018}  (which was introduced by Sommerfeld \cite{Sommerfeld} in 1912)
excludes inwardly propagated waves to characterize the outgoing nature of the radiated wave \cite{CK} and thus
guarantees the uniqueness of solutions to the system \eqref{eq:1-cornscatter2018}.
This radiation condition gives arise to the asymptotic expansion of $u$ at the infinity
	\begin{equation} \label{eq:farfield-cornscatter2018}
	u(x) = \frac{e^{\mathrm{i}k |x|}}{\sqrt {|x|}} u^\infty(\hat x) + \mathcal{O} ( |x|^{- 3 / 2} ) \quad \mbox{as} \ \ \ |x| \rightarrow \infty,
	\end{equation}
which is uniform in all directions
$\hat x := x/{|x|} \in \mathbb{S} :=\{ x=(x_1,x_2) \in \R^2:  x_1^2 + x_2^2 = 1 \}$.
The function  $u^\infty(\hat{x})$ is referred to as the far-field pattern (or scattering amplitude). It is an analytic function with respect to the observation direction $\hat{x}\in \mathbb{S}$ (see \cite{CK}). 
	
Using the variational approach, it is easy to prove that the system \eqref{eq:1-cornscatter2018} admits a unique solution in $H^2_{loc}(\R^2)$; see \cite[Chapter 5]{CK} or \cite[Chapter 5]{Cakoni}. Since the far-field pattern encodes information of the source,   we are interested in the inverse problem of recovering the source support $\partial D$ and/or the source term $f$ from the far-field pattern over all observation directions at a fixed frequency.  Throughout the paper we suppose that the refractive index function $n(x)$ is known and $k>0$ is an arbitrarily fixed wave number. Note that the relation between $\partial D$ and $u^\infty$ is non-linear, but the operator mapping $f$ to $u^\infty$ is linear.

It is well known that a single far-field pattern cannot uniquely determine a source function (even its support) in general, due to the existence of non-radiating sources, for instance, $f_0:=(\Delta+k^2n(x))\varphi$ where $\varphi\in C_0^\infty(\R^2)$. \rot{In fact, by linear superposition, it is easy to conclude that $f$ and $f+f_0$ could radiate identical far-field data at infinity.}  It was proved in \cite{EV, BLT} that the far-field data over a range of frequencies can be used to uniquely determine \rot{a spatially-dependent} source term compactly supported in a homogeneous medium. Similar uniqueness results can be proved \rot{if} the background medium is inhomogeneous but known in advance; we refer to  \cite{HKZ2019} for recent uniqueness results for various inverse source problems in the time domain.

In \cite{Ik1999}, Ikehata proposed the enclosure method to recover the convex hull of a polygonal source \rot{in a homogeneous medium using} a single Cauchy pair of near-field data. Uniqueness was also verified as a by-product of the enclosure method, if the source function is H\"older continuous and \rot{non-vanishing near corners}. For source functions without a polygonal support, Kusiak and Sylvester \cite{KS} introduced the concept "scattering support" to define the minimal set that supports the scattered field; see also \cite{Sy2006}. In a recent paper \cite{Bl2018}, the author reveals that the far-field data determine not only the convex-polygonal support but also the source values at corner points. Moreover, it is proved that non-radiating sources must vanish at corner points. Other related studies on corner scattering are devoted to the absence of real non-scattering energies \cite{BLS,PSV,ElHu2015, HSV, ElHu2018acoustic}  and uniqueness with a single measurement data in shape identification arising from inverse conductivity \cite{Ik2000, Seo,FI89} and inverse medium scattering problems \cite{Ik1999,Ik2005, ElHu2015, ElHu2018acoustic, HSV, HLZ, LHY} .

	

This paper has generalized the existing results in the literature in several aspects,  providing new insights into inverse source problems with a single measurement data. The major novelties are summarized as follows. First, we consider the inverse source problem with a single far-field pattern in an inhomogeneous background medium rather than a homogeneous one; Second,
  we prove that the gradient of $C^{1,\alpha}$-smooth source terms at corner points can be uniquely identified, in addition to the information of source values at corners and the convex-polygonal support which were already discussed in \cite{Ik1999,KS,Bl2018}. Moreover, we show that a convex-polygonal source support can be uniquely identified under the more general assumption that the lowest order expansion of the source function is harmonic near corner points (see Corollary \ref{cor3} and the remarks thereafter);  Thirdly, we define an admissible set of source functions (including harmonic functions) which can be uniquely identified by a single far-field pattern.
Our approach relies heavily on the singularity analysis of
the inhomogeneous Laplace equation in a corner domain, which was already used in  \cite{ElHu2018acoustic} for treating inverse medium scattering problems. Our arguments are presented in a two-dimensional setting for simplification, however they can be carried over to curvilinear polyhedral sources in 3D following the lines of \cite[Sections 5]{ElHu2018acoustic}.	
	
	
Below we state the main uniqueness results of this paper. Denote by $B_a(z)=\{x\in \R^2: |x-z|<a\}$ the disk centered at $z\in \R^2$ with radius $a>0$. 
	\begin{thm}[Determination of source support] \label{thm:Holdercontin-cornscatter2018}
	Suppose that
	\begin{itemize}
	\item[(i)] $D := \supp (f) \subset \R^2$ is a convex polygon.
	\item[(ii)] For each corner point $O\in\partial D$,  it holds that
	$f \in C^{1,\alpha}(\overline{D \cap B_\epsilon(O)})$ for some $\epsilon>0$, $0 < \alpha <1$ and that $|f(O)| + |\nabla f(O)| > 0$.
		\end{itemize}
Then the source support  $\partial D$ together with $f$ and $\nabla f$ at each corner point can be uniquely determined by $u^\infty(\hat{x})$ for all $\hat{x}\in\mathbb{S}$.
	\end{thm}
\rot{The proof of Theorem \ref{thm:Holdercontin-cornscatter2018} implies the following weak result which was proved in \cite{Ik1999,KS,Bl2018}}.

\begin{cor}\label{cor1} If the second condition in Theorem \ref{thm:Holdercontin-cornscatter2018} is replaced by
$f \in C^{0,\alpha}(\overline{D \cap B_\epsilon(O)})$ for some $\epsilon>0$, $0 < \alpha<1$ and $|f(O)| > 0$
at the corner point $O\in\partial D$. Then the source support  $\partial D$ and $f(O)$ \rot{at each corner point $O\in\partial D$}  can be uniquely determined by $u^\infty$.
\end{cor}

Theorem \ref{thm:Holdercontin-cornscatter2018}, Corollary \ref{cor1} together with Corollary \ref{prop:smooth-cornscatter2018} (see Section \ref{sec:contiSource-cornscatter2018}) state that a source support of convex-polygon type and partial information (derivatives of the zeroth and first orders) of the source function on corner points can be uniquely identified.
In addition to these boundary information,  we prove that the entire source function from an admissible class can be also uniquely determined by a single far-field pattern.

 Given $A(x)=(a_1(x), a_2(x))\in (L^\infty(B_R))^2$ and $b\in L^\infty(B_R)$,  introduce the admissible set
\begin{equation} \label{eq:HSource-cornscatter2018}
S(A,b):=\{v(x): \Delta v(x) + A(x) \cdot \nabla v(x)+ b(x)\, v(x) = 0 \quad \text{in}~ B_R\}.
\end{equation}
In the special case that $A\equiv 0$ and $b\equiv 0$, the set $S(A,b)$ consists of all harmonic functions in $B_R$.
Theorem \ref{thm:analytic-cornscatter2018} below states that a single far-field pattern can be used to determine a source term given by the restriction to a convex polygon of a special function from the admissible set $S(A, b)$.
	\begin{thm}[Determination of source terms] \label{thm:analytic-cornscatter2018}
	Assume that $D=\text{supp}(f)$ is a (non-empty) convex polygon and that $f=v|_{\overline{D}}$  for some $v\in S(A, b)$. Suppose further that $A$ and $b$ are known functions which are analytic near each corner of $\partial D$.
Then the source term $f$ (and also its support $D$) can be uniquely determined by $u^\infty(\hat{x})$ for all $\hat{x} \in \mathbb{S}$.
	\end{thm}
	
	For the multi-index $\beta=(\beta_1, \beta_2)$, we say $\beta\geq 0$ if $\beta_j \geq 0$ for $j=1,2$. If $\beta_j\in \N_0:=\{0, 1,2, \cdots\}$,
we define the differential operator
	$\partial^\beta u:=\partial_{x_1}^{\beta_1} \partial_{x_2}^{\beta_2} u$.
The source term $f$ is called non-radiating if the resulting far-field pattern vanishes identically. Below we describe a class of radiating sources in a  piecewise homogeneous background medium. 	
Here the source support is not necessarily a convex polygon. 
	\begin{cor}[Characterization of radiating sources] \label{thm:NonRadiatingSource-cornscatter2018} Let $n(x)\equiv n_0\neq 1$ in $B_R$. Suppose that the support $D=\text{supp}(f)\neq \emptyset$ contains at least one corner point $O\in \partial D$ and its exterior $\R^2\backslash\overline{D}$ is connected.  Then $f$ is a radiating source under one of the following conditions
		\begin{enumerate}
			\item[(i)] There exists some $l\in \N_0$ and $\alpha\in(0,1)$ such that $f \in C^{l+1,\alpha}(\overline{D \cap B_\epsilon(O)}) \cap W^{l,\infty}(B_\epsilon(O))$, and for some
			multi-index $\beta=(\beta_1, \beta_2)$  ($\beta_j\in \N_0$) with $|\beta| := \beta_1+\beta_2=l$ it holds that
\begin{equation}\label{eq:assumption}
| \partial^\beta f(O)|+|\partial^{\beta'} f(O)|> 0,\quad \beta' -\beta\geq 0,\, |\beta' - \beta| = 1.
\end{equation}
			\item[(ii)] $f=v|_{\overline{D}}$ for some $v \in S(A,b)$ where $A$ and $b$ are analytic functions near $O$.
		\end{enumerate}		
	\end{cor}
	Condition \rot{(\ref{eq:assumption})} has been used  in \cite{ElHu2018acoustic} to describe medium discontinuity in a low contrast case.
	In the special case of $l=0$,
 the first condition in Corollary \ref{thm:NonRadiatingSource-cornscatter2018} is equivalent to the condition (ii) of Theorem \ref{thm:Holdercontin-cornscatter2018} at the corner $O$; see also Remark \ref{rem3.1}.  Corollary \ref{thm:NonRadiatingSource-cornscatter2018} (i) in the special case $l=0$ implies that the zeroth and first order derivatives of
 a non-radiating source must vanish at the corner point (cf. Theorem \ref{thm:Holdercontin-cornscatter2018}).

As a corollary of the proofs of Theorems \ref{thm:Holdercontin-cornscatter2018} and \ref{thm:analytic-cornscatter2018}, we claim that the radiated field must be "singular" (that is, non-analytic) at corner points; see Corollary \ref{cor4} below. This excludes the possibility of analytical extension in a corner domain for solutions of forward source problems, which is important in designing inversion algorithms with a single measurement data; see e.g. the enclosure method \cite{Ik1999}, the range test approach  \cite{KPS, KS} as well as \cite[Chapter 5]{NP2013} and \cite{EH2019}.
\begin{cor}[Absence of extension at corner points]\label{cor4} Let $n(x)\equiv n_0\neq 1$ in $B_R$. Suppose that $O\in \partial D$ is a corner point lying on the boundary of $D=\text{supp}(f)\neq \emptyset$  and one of the conditions (i)-(ii) in Corollary \ref{thm:NonRadiatingSource-cornscatter2018} holds at $O$. Then $u$ cannot be analytically extended from $B_R\backslash\overline{D}$ to the interior of $D$ across the corner point $O$. In other words, $u$ cannot be analytic at $O$.
\end{cor}
\rot{In an inhomogeneous background medium, $u$ maybe not analytic in $\R^2\backslash\overline{D}$ if
 $n(x)\in C^{0,\alpha}(\R^2)$. However,
Corollary \ref{cor4} still holds if we properly define the \emph{extension} across a corner point.} Corollary \ref{cor4} motivates us to design a data-driven and domain-defined sampling scheme for imaging a convex polygonal source support (see Section \ref{inversion} for more details).

From the proofs of Theorems \ref{thm:Holdercontin-cornscatter2018} and  \ref{thm:analytic-cornscatter2018}, we can easily show the following results.
\begin{cor} \label{cor3}
\begin{itemize}
\item[(i)] Assume that $D=\text{supp}(f)$ is a (non-empty) convex polygon. For  each corner point $O_j\in\partial D$, suppose that
$f\in C^{0, \alpha_j}(D\cap B_\epsilon(O_j))$ ($\alpha_j\in (0,  1)$) has the asymptotic behavior
\begin{equation} \label{eq:f}
		f(x) = r^{N_j}(A_j \cos N_j \theta + B_j \sin N_j\theta) +  o(r^{N_j}), \quad |x| \to 0,\,
		\end{equation}
		for some $N_j\in \N_0$ and $A_j, B_j\in\mathbb C$ with $|A_j|+|B_j|>0$,
uniformly in all directions $\hat{x}=x/|x|$ such that $x\in D\cap B_\epsilon(O)$. Then $\partial D$ can be uniquely determined by a single far-field pattern $u^\infty$.
\item[(ii)]  The results in Corollaries \ref{thm:NonRadiatingSource-cornscatter2018} and  \ref{cor4} hold true if the lowest order expansion of $f$ around the corner takes the form \eqref{eq:f}.
\end{itemize}
\end{cor}
	We remark that the opening angle of the corner domain in Corollaries \ref{thm:NonRadiatingSource-cornscatter2018},  \ref{cor4} and Corollary \ref{cor3} (ii) is allowed to be any number in $(0, 2\pi)\backslash\{\pi\}$. The condition (\ref{eq:f}) in Corollary (\ref{cor3})	covers those assumptions made in Theorem \ref{thm:Holdercontin-cornscatter2018}, Corollary \ref{cor1} and Theorem \ref{thm:analytic-cornscatter2018} for identifying a source support.

The rest of the paper is organized as follows. In Section \ref{sec:preli-cornscatter2018}, we give some preliminaries of weighted H\"older spaces in an infinite sector and prove an important Lemma \ref{lem:pre-cornscatter2018}, which plays the key role to all of our uniqueness proofs. Section \ref{sec:contiSource-cornscatter2018} presents uniqueness results for recovering a convex-polygonal source support. Section \ref{sec:harmoSource-cornscatter2018} is devoted to the uniqueness in recovering source terms. In Section \ref{sec:radiatS-cornscatter2018}, we provide the proofs of Corollaries \ref{thm:NonRadiatingSource-cornscatter2018},  \ref{cor4} and \ref{cor3}. The reconstruction scheme will be described and compared with other approaches in Section \ref{inversion}.

\section{Preliminaries} \label{sec:preli-cornscatter2018}
In this section, we present some notations and auxiliary lemmas for weighted H\"older spaces in an infinite sector.
Denote by $O=(0,0)$ the origin in $\R^2$. Let $(r, \theta)$ be the polar coordinates of $x=(x_1, x_2)\in \R^2$.
Define $\Sigma =\Sigma_{\theta_0}:=\{ (r,\theta) ; r >0, |\theta| < \theta_0/2 \}$, an infinite sector in $\R^2$ with the opening angle $\theta_0 \in (0,\pi)$ and with the vertex $O$ located at the origin.
For $j \in \N_0$, we introduce $\nabla^j$ as the vector of all partial differential operators of order $j$ with respect to the spatial variable, that is
	\[
	\nabla^j \varphi (x) := \big\{ \partial_{x_1}^{j_1} \partial_{x_2}^{j_2} \varphi(x) \,;\, j_1,j_2 \in \N_0, j_1 + j_2 = j \big\}.
	\]
For $0 < \alpha < 1$, $l\in \N_0$ and $\beta\in \R$, we define the weighted H\"older spaces  $\Lambda_\beta^{l,\alpha} (\Sigma)$
 endowed with the norm  (see \cite{Kondratiev, KMR, MNP, NP})
	\begin{equation} \label{eq:FuncNOrm-cornscatter2018}
	\nrm[\Lambda_\beta^{l,\alpha} (\Sigma)]{\varphi}
	:=
	\sup_{x \in \Sigma} \left\{\sum_{j=0}^l |x|^{\beta-\alpha-l+j} |\nabla^j \varphi(x)|\right\}
	+
	\sup_{x,y \in \Sigma} \left\{\frac {\big| |x|^\beta \nabla^l \varphi(x) - |y|^\beta \nabla^l \varphi(y) \big|} {|x-y|^\alpha}\right\}.
	\end{equation}
	We may denote  $\Lambda_\beta^{l,\alpha} (\Sigma)$ as $\Lambda_\beta^{l,\alpha}$ if it is clear in the context.
	Obviously, any function from $\Lambda_\beta^{l,\alpha} (\Sigma)$ lies in $C^{l,\alpha}(\Sigma\cap \{(r,\theta): a<r<b\})$ for some $0<a<b$ and the subscript $\beta$ characterizes the asymptotic behaviour at the origin and at infinity.	
It can be verified that if $u \in \Lambda_\beta^{l,\alpha}$, then $\nabla^j u \in \Lambda_\beta^{l-j,\alpha}$ for all $j = 0,1,\cdots,l$.
In our applications we shall only care about the asymptotic behavior near the corner point.
For any function $u\in  \Lambda_\beta^{l,\alpha}(\Sigma)$ with a compact support in $\Sigma\cap B_R(O)$ ($R>0$), we have
\ben
\nabla^{j}u(x)=O(r^{l-j+\alpha-\beta})\qquad\mbox{as}\quad r\rightarrow 0^+\quad\mbox{in}\quad\Sigma,
\enn for $j=0,1\cdots, l$.
In particular, we have the decaying rate $u(x)=O(r^{l+\alpha-\beta})$  if $\beta<l+\alpha$ and the inclusion relation
	\begin{equation} \label{eq:inclusion-cornscatter2018}
	\Lambda_\beta^{l,\alpha} (\Sigma)\subset \Lambda_{\beta+1}^{l,\alpha}(\Sigma).
	\end{equation}

In the following two lemmas we prove some properties of  $\Lambda_\beta^{l,\alpha} (\Sigma)$ by definition.
	
		\begin{lem} \label{lem:product-cornscatter2018}
		Let $\beta \in \R$ and $\alpha,\, \alpha' \in (0,1)$ with $\alpha' \geq \alpha$ and let $\Sigma$ be an infinite sector.  Assume that $g \in \Lambda_\beta^{0,\alpha}(\Sigma)$ with compact support and that $f \in C^{0,\alpha'}(\Sigma)$. Then the product $f g$ belongs to the space $\Lambda_\beta^{0,\alpha}(\Sigma)$.
	\end{lem}
	\begin{proof} We need to prove the boundedness of the norm
		\begin{align}
			\nrm[\Lambda_\beta^{0,\alpha} (\Sigma)]{f g}
			& \ =
			\sup_{x \in \Sigma} |x|^{\beta-\alpha} |f(x) g(x)|
			+
			\sup_{x,y \in \Sigma} \frac {\big| |x|^\beta f(x) g(x) - |y|^\beta f(y) g(y) \big|} {|x-y|^\alpha} \nonumber\\
			& := \mathcal I_1 + \mathcal I_2. \label{eq:product.0-cornscatter2018}
		\end{align}
		The first term $\mathcal I_1$ in \eqref{eq:product.0-cornscatter2018} can be bounded by
		\begin{equation}
		\mathcal I_1
		\leq
		\sup_{x \in \Sigma} |f(x)| \cdot \sup_{x \in \Sigma} |x|^{\beta-\alpha} |g(x)|
		\leq
		\nrm[C^{0,\alpha}(\Sigma)]{f} \cdot \nrm[\Lambda_\beta^{0,\alpha}(\Sigma)]{g} < +\infty. \label{eq:product.1-cornscatter2018}
		\end{equation}
		Using the compactness of $\supp(g)$ and the triangle inequality, we can estimated the second term $\mathcal I_2$ by
		\begin{align}
		\mathcal I_2
		& \leq
		\sup_{x,y \in \Sigma} \frac {\big| |x|^\beta f(x) g(x) - |y|^\beta f(x) g(y) \big|} {|x-y|^\alpha}
		+
		\sup_{x,y \in \Sigma} \frac {\big| |y|^\beta f(x) g(y) - |y|^\beta f(y) g(y) \big|} {|x-y|^\alpha} \nonumber\\
		& \leq
		\sup_{x \in \Sigma} |f(x)| \cdot \sup_{x,y \in \Sigma} \frac {\big| |x|^\beta g(x) - |y|^\beta g(y) \big|} {|x-y|^\alpha}\\
		&\quad +
		\sup_{y \in \supp g} |y|^\alpha \cdot \sup_{y \in \Sigma} |y|^{\beta - \alpha} |g(y)| \cdot \sup_{x,y \in \Sigma} \frac {\big| f(x) - f(y) \big|} {|x-y|^\alpha} \nonumber\\
		& \leq \nrm[C^{0,\alpha'}(\Sigma)]{f} \cdot \nrm[\Lambda_\beta^{0,\alpha}(\Sigma)]{g} + C \nrm[\Lambda_\beta^{0,\alpha}(\Sigma)]{g} \cdot \nrm[C^{0,\alpha'}(\Sigma)]{f} \nonumber\\
		& < +\infty. \label{eq:product.2-cornscatter2018}
		\end{align}
		Combining \eqref{eq:product.0-cornscatter2018}-\eqref{eq:product.2-cornscatter2018}, we arrive at our conclusion.	
	\end{proof}

	\begin{lem} \label{lem:embding-cornscatter2018}
	Any compactly supported function in $\Lambda_{-N+1}^{2,\alpha'}(\Sigma)$ ($N\in \N_0$) belongs to $\Lambda_{-N}^{0,\alpha'}(\Sigma)$.
	\end{lem}
	\begin{proof}
		Suppose that $\varphi \in \Lambda_{-N+1}^{2,\alpha'}(\Sigma)$ has a compact support in $\Sigma$. By the definition of the norm \eqref{eq:FuncNOrm-cornscatter2018}, we obtain
		\begin{equation} \label{eq:embd1-cornscatter2018}
		\mathcal M_0 := \sup_{x \in \Sigma} |x|^{-N-1-\alpha'} |\varphi(x)| < +\infty \quad\text{and}\quad \mathcal M_1 := \sup_{x \in \Sigma} |x|^{-N-\alpha'} |\nabla \varphi(x)| < +\infty.
		\end{equation}
		To prove $\varphi \in \Lambda_{-N}^{0,\alpha'}(\Sigma)$, we only need to prove the boundedness of		
		\begin{equation} \label{eq:embd4-cornscatter2018}
				\mathcal M_2
		:=
		\sup_{x \in \Sigma} |x|^{-N-\alpha'} |\varphi(x)|,
		\quad
		\mathcal M_3
		:=
		\sup_{x,y \in \Sigma} \frac {\big| |x|^{-N} \varphi(x) - |y|^{-N} \varphi(y) \big|} {|x-y|^{\alpha'}}.
		\end{equation}
		
	We first prove $\mathcal M_3<\infty$. Write $\psi(x) := |x|^{-N} \varphi(x)$. The derivative of $\psi$ can be estimated by		\begin{align*}
		|\partial_{x_j} \psi(x)|
		& = \big| -N |x|^{-N-1} \frac {x_j} {|x|} \varphi(x) + |x|^{-N} \partial_{x_j} \varphi(x) \big| \nonumber\\
		& \leq |N| |x|^{-N-1} |\varphi(x)| + |x|^{-N} |\partial_{x_j} \varphi(x)| \nonumber\\
		& \leq |N| |x|^{\alpha'} \mathcal M_0 + |x|^{\alpha'} \mathcal M_1,
		\end{align*} for all $x\in\mbox{supp}(\varphi) $ and  $j=1,2$, implying that
\begin{equation} \label{eq:embd2-cornscatter2018}
		|\nabla \psi(x)| \leq C (\mathcal M_0 + \mathcal M_1) < +\infty \quad \text{for all} ~ x \in \Sigma \cap \supp \varphi.
		\end{equation}
		Now we can estimate $\mathcal M_3$ by applying the mean-value theorem
		\begin{align}\label{eq:embd222-cornscatter2018}
		\mathcal M_3
		 = \sup_{x,y \in \Sigma} \frac {\big| \psi(x) - \psi(y) \big|} {|x-y|^{\alpha'}}
		 \leq ||\nabla \psi||_{L^\infty(\Sigma \cap \supp \varphi )} \sup_{x,y \in \Sigma \cap \supp \varphi} |x-y|^{1-\alpha'}<\infty.
		\end{align}
Writing $\mathcal M_2$ as $\mathcal M_2 = \sup_{x \in \Sigma} \big( |x| \cdot |x|^{-N-1-\alpha'} |\varphi(x)| \big)$ and using the compactness of $\supp \varphi$, we get
		\begin{equation} \label{eq:embd3-cornscatter2018}
		\mathcal M_2 \leq \big( \sup_{x \in \Sigma \cap \supp \varphi} |x| \big) \cdot \sup_{x \in \Sigma} \big( |x|^{-N-1-\alpha'} |\varphi(x)| \big) \leq C \mathcal M_0 < +\infty.
		\end{equation}
Combining \eqref{eq:embd222-cornscatter2018} and \eqref{eq:embd3-cornscatter2018} we conclude that $\varphi \in \Lambda_{-N}^{0,\alpha'}(\Sigma)$. The proof is complete.
	\end{proof}
	
	The proofs of our main results depend on
	Lemma \ref{lem:pre-cornscatter2018} below, which is motivated by Propositions 10, 12 and 13 of \cite{ElHu2018acoustic}.
	Introduce the finite sector
	$$\Sigma_\epsilon := \{ (r,\theta) \in \R^2: \,\, 0 < r < \epsilon,\, 0 < \theta < \theta_0 \},\quad \theta_0\in (0, \pi), \quad 0<\epsilon<1$$ and its partial boundary
	$$
	\Gamma_\epsilon:=\{ (r,\theta) \in \R^2 \,:\, 0 < r < \epsilon,\, \theta=0, \theta_0 \}. $$
	
	\begin{lem} \label{lem:pre-cornscatter2018}
		Suppose  $q, h \in C^{0,\rot{\alpha}}(\overline{\Sigma_\epsilon})$ for some
		$\rot{\alpha}\in(0,1)$ and that
		 the \rot{lowest order expansion} of $h$ as $|x|\rightarrow 0$ is harmonic, that is, the asymptotic expansion
		\begin{equation} \label{eq:hform1-cornscatter2018}
		h(x) = r^N\left(A \cos (N \theta) + B \sin (N\theta)\right) + \mathcal O(r^{N+\alpha}), \quad |x| \to 0,\, x \in \Sigma_\epsilon
		\end{equation}
holds uniformly in all $\theta\in(0, \theta_0)$ for some
 $N\in\mathbb N_0$ and $A$, $B \in \mathbb C$. Let $u \in H^2(\Sigma_\epsilon)$ be a solution to the boundary value problem
		\begin{equation} \label{eq:BVPh-cornscatter2018}
		\mbox{(BVP)}:\quad \begin{cases}
		\Delta u + q(x) u = h(x) & \mbox{in } \quad\Sigma_\epsilon, \\
		u = \partial_\nu u = 0 & \mbox{on }\quad \Gamma_\epsilon,
		\end{cases}
		\end{equation}
		where  $\partial_\nu u$ denotes the normal derivative of $u$. Then it holds that $A = B = 0$.
	\end{lem}

	The proof of Lemma \ref{lem:pre-cornscatter2018} is based on a precise characterization of the singularity of solutions to the inhomogeneous Laplacian equation in a corner domain (see, e.g., \cite[Chapters 2 and 3]{NP}). Our argument is a refine of the proof of \rot{\cite[Lemma 2]{ElHu2018acoustic}} under the assumption (\ref{eq:hform1-cornscatter2018}). Below we sketch the proof for the readers' \rot{convenience and refer to  \cite[Lemma 2]{ElHu2018acoustic}} for more details. 	
	\begin{proof}[Proof of Lemma \ref{lem:pre-cornscatter2018}]
	Introduce the cutoff function $\chi(r) \in C_0^\infty(\R)$ satisfying $\chi(r) \equiv 1$ when $r< \frac \epsilon 2$ and $\chi(r) \equiv 0$ when $r> \epsilon$.
		From \eqref{eq:BVPh-cornscatter2018} we deduce the inhomogeneous Laplacian equation defined over the infinite sector $\Sigma$:
		\begin{equation} \label{eq:PDEsChi-cornscatter2018}
		\Delta (\chi u) = \chi h - (\chi u)q + [\Delta,\chi] u=:f, \qquad \text{in}~ \Sigma,
		\end{equation}
		with the commutator operator
		\begin{equation*} 
		[\Delta,\chi] u := \Delta (\chi u) - \chi \Delta u = 2 \nabla \chi \cdot \nabla u + (\Delta \chi) u.
		\end{equation*}
		We are going to analyze the regularity of $\chi u$ (and thus $u$ itself) in a neighboring hood of the origin by using the vanishing of the Cauchy data of $\chi u$  on $\theta=0, \theta_0$ together with the decaying rate of $f$ near $O$. For clarity we shall divide the rest of the proof into four steps.
				
		\noindent {\bf Step 1}: Show that the right hand side of \eqref{eq:PDEsChi-cornscatter2018}  satisfies $f\in\Lambda_1^{0,\rot{\alpha}}(\Sigma)$.
		
		By the assumption of $h$ and the inclusion relation \eqref{eq:inclusion-cornscatter2018} for compactly supported functions, we obtain
		\begin{equation} \label{eq:hBelongs-cornscatter2018}
		 \chi h \in \Lambda_{-j}^{0,\rot{\alpha}}(\Sigma), \quad\text{for every}\quad j = -1,0,1,\cdots, N-1.
		\end{equation}
		The assumption $u \in H^2(\rot{\Sigma_\epsilon})$ implies
		\begin{equation} \label{eq:quStart-cornscatter2018}
		\chi\,u \in \Lambda_1^{0,\rot{\alpha}}(\Sigma). \quad \text{and} \quad (\chi u)\,q \in C^{0,\rot{\alpha}}(\Sigma) \subseteq \Lambda_1^{0,\rot{\alpha}}(\Sigma).		\end{equation}
We now consider the regularity of $[\Delta,\chi] u$ appearing in \eqref{eq:PDEsChi-cornscatter2018}.
	Since the supports of $\nabla \chi$ and $\Delta \chi$ are contained in ${\{ x \in \R^2 \,:\, \frac \epsilon 2 \leq |x| \leq \epsilon \}}$, it follows from \eqref{eq:BVPh-cornscatter2018} that
		\begin{equation} \label{eq:DeltauPDE-cornscatter2018}
		\Delta u = h(x) - q(x) u \quad \mbox{in } \Sigma' _\epsilon:= \Sigma_\epsilon \cap \{ \frac \epsilon 2 < |x| < \epsilon \}.
		\end{equation}
Since	$h \in C^{0,\rot{\alpha}}(\Sigma'_\epsilon),
		qu \in C^{0,\rot{\alpha}}(\Sigma'_\epsilon)$,
	by standard elliptic regularity theory it holds that
		\begin{equation} \label{eq:u2alpha.2-cornscatter2018}
		u \in C^{2,\rot{\alpha}}(\Sigma'_\epsilon),
		\end{equation}
		implying the relation
		\begin{equation} \label{eq:uTail-cornscatter2018}
		[\Delta,\chi] u \in \Lambda_{-j}^{0,\rot{\alpha}}(\Sigma) \quad\text{for any integer}~ j.
		\end{equation}

		\noindent {\bf Step 2}: Prove $\chi u \in \Lambda_{-N}^{0,\rot{\alpha}}(\Sigma)$.
		
		We can summarize from \eqref{eq:hBelongs-cornscatter2018}, \eqref{eq:quStart-cornscatter2018} and \eqref{eq:uTail-cornscatter2018} that
		\begin{equation} \label{eq:hBaaaelongs-cornscatter2018}
		\left\{\begin{aligned}
		\chi h & \in \Lambda_j^{0,\rot{\alpha}}(\Sigma), \quad\text{for every}~ j = 1,0,-1,\cdots, -N+1, \\
		(\chi u)\,q & \in \Lambda_\ell^{0,\rot{\alpha}}(\Sigma), \quad\text{where}\; \ell=1,\\
		[\Delta,\chi] u & \in \Lambda_j^{0,\rot{\alpha}}(\Sigma), \quad\text{for any integer}~ j.
		\end{aligned}\right.
		\end{equation}
		This implies that the right-hand-side of \eqref{eq:PDEsChi-cornscatter2018} belongs to $f\in {\Lambda_1^{0,\rot{\alpha}}(\Sigma)}$.
\rot{By the solvability of the Laplace equation in an infinite sector}
 (see e.g. \cite[Theorem 6.11, Chapter 3]{NP}), we obtain
		\begin{equation} \label{eq:chiu2-cornscatter2018}
			\chi u \in \Lambda_1^{2,\rot{\alpha}}(\Sigma).
		\end{equation}
Now, applying Lemma \ref{lem:embding-cornscatter2018} gives $\chi u \in \Lambda_{0}^{0,\rot{\alpha}}(\Sigma)$, which improves the subscript $\beta$ in the first step (cf. \ref{eq:quStart-cornscatter2018}) from $\beta=1$ to $\beta=0$. This means that $u$ at the corner point is getting less singular. \rot{Moreover, combining the fact $\chi u \in \Lambda_{0}^{0,\rot{\alpha}}(\Sigma)$}  with Lemma \ref{lem:product-cornscatter2018} leads to 		
	those relations in \eqref{eq:hBaaaelongs-cornscatter2018} where the subscript $\ell$ is replaced by $0$, which in turn gives $\chi u \in \Lambda_0^{2,\rot{\alpha}}(\Sigma)$.	
	Repeating this process, we finally arrive at
\begin{equation} \label{eq:uBelongs-cornscatter2018}
	f\in\Lambda_{-N+1}^{0,\rot{\alpha}}(\Sigma),\quad\chi u \in \Lambda_{-N+1}^{2,\rot{\alpha}}(\Sigma)\subset
	\Lambda_{-N}^{0,\rot{\alpha}}(\Sigma').
		\end{equation}

		\noindent {\bf Step 3}: \rot{Singularity analysis of (\ref{eq:PDEsChi-cornscatter2018}).}
		
		Write $h_1(x) := f- r^N(A \cos N \theta + B \sin N\theta) \chi$. From \eqref{eq:BVPh-cornscatter2018} and the definition of $\chi$, it follows that the function $\chi u$ solves the following Dirichlet and Neumann boundary value problems:
		\begin{equation} \label{eq:Deltauhh1D-cornscatter2018}
		\text{(Dirichlet BVP)}~
		\begin{cases}
		\Delta (\chi u) = r^N(A \cos N \theta + B \sin N\theta)\chi + h_1(x) & \mbox{in } \Sigma \\
		(\chi u) = 0 & \mbox{on } \partial \Sigma,
		\end{cases}
		\end{equation}
		\begin{equation} \label{eq:Deltauhh1N-cornscatter2018}
		\text{(Neumann BVP)}~
		\begin{cases}
		\Delta (\chi u) = r^N(A \cos N \theta + B \sin N\theta) + h_1(x) & \mbox{in } \Sigma \\
		\partial_\nu (\chi u) = 0 & \mbox{on } \partial \Sigma.
		\end{cases}
		\end{equation}
Note that $r^N(A \cos N \theta + B \sin N\theta) \chi\in \Lambda_{-N+1}^{0,\rot{\alpha}}(\Sigma)$.  From \eqref{eq:hform1-cornscatter2018}, we see
\ben
\chi \left[h- r^N(A \cos N \theta + B \sin N\theta) \right] \in \Lambda_{-N}^{0,\rot{\alpha}}(\Sigma).
\enn
Together with
 \eqref{eq:uTail-cornscatter2018} and \eqref{eq:uBelongs-cornscatter2018}, this gives  $h_1\in \Lambda_{-N}^{0,\alpha'}(\Sigma) \subset \Lambda_{-N+1}^{0,\alpha'}(\Sigma)$ for all $0 < \alpha' < \alpha$. Therefore, from [Theorem 6.11, Chapter 3, \citen{NP}], the Dirichlet BVP \eqref{eq:Deltauhh1D-cornscatter2018} admits a unique solution in $\Lambda_{-N+1}^{2,\alpha'}(\Sigma)$ if $\alpha' \neq \frac {j\pi} {\theta_0} - 1 - N\, (j\in \mathbb{Z})$, where $\theta_0$ is the opening angle of the sector $\Sigma$. Further,  using  \cite[Proposition 5]{ElHu2018acoustic}, \cite[Proposition 2.12, Chapter 2]{NP} and the arbitrariness of $\alpha'\in(0, \alpha) $,
  the function $\chi u\in  \Lambda_{-N}^{0,\alpha'}(\Sigma)  $  around the corner can be decomposed into two parts:
 $\chi u = u_{\mathcal D}^{(1)} + u_{\mathcal D}^{(2)}$ where
		\begin{align*}
		u_{\mathcal D}^{(1)} & = q_{\mathcal D,N+2} + C_{\mathcal D}\, r^{N+2}[\ln r \sin(N+2)\theta + \theta \cos(N+2)\theta], \quad C_{\mathcal D}\in \mathbb C,
		\nonumber \\
		u_{\mathcal D}^{(2)} & = d_{\mathcal D}r^{N+2} \sin(N+2)\theta+\sum_{j\in I(\theta_0, N)} d_{{\mathcal D},j}\, r^{\frac {j\pi} {\theta_0}} \sin (\frac {j\pi} {\theta_0} \theta) + \mathcal O(r^{N+2+\alpha'}),\quad d_{\mathcal D},
	d_{{\mathcal D},j}\in \mathbb C, 	
		\end{align*}
		where $$I(\theta_0, N):=\{j: \frac {j\pi} {\theta_0} \in (N+1, N+2)\}, \quad
		d_{\mathcal D}=0\quad\mbox{if}\; (N+2)\theta_0/\pi\notin \N,
		$$ and $q_{\mathcal D,N+2}$ is a polynomial of order $N+2$ satisfying
		\be \label{eq:qDN2-cornscatter2018}
		\Delta q_{\mathcal D,N+2} = r^N(A \cos N \theta + B \sin N\theta)\quad \mbox{in}\quad\Sigma,\qquad
		q_{\mathcal D,N+2} =0 \quad\mbox{on}\quad \partial \Sigma.		
		\en
	Similarly, by \cite[Chapter 2, Proposition 2.12]{NP} the Neumann BVP \eqref{eq:Deltauhh1N-cornscatter2018} admits a unique solution $\chi u= u_{\mathcal N}^{(1)} + u_{\mathcal N}^{(2)}\in\Lambda_{-N+1}^{2,\alpha'}(\Sigma) \subset  \Lambda_{-N}^{0,\alpha'}(\Sigma)  $ where
		\begin{align}
& u_{\mathcal N}^{(1)}(x)= q_{{\mathcal N},N+2}(x) + C_{\mathcal N} r^{N+2}[\ln r \cos(N+2)\theta - \theta \sin(N+2)\theta], \quad C_{\mathcal N}\in \mathbb C,    \nonumber\\ \nonumber
 & u_{\mathcal N}^{(2)}(x)= d_{\mathcal N}r^{N+2} \cos(N+2)\theta+ \sum_{\substack{j \in I(\theta_0, N)}} d_{{\mathcal N},j}\, r^{\frac {j\pi} {\theta_0}} \cos (\frac {j\pi} {\theta_0} \theta) + \mathcal O(r^{N+2+\alpha'}), \quad d_{\mathcal N}, d_{{\mathcal N},j}\in \mathbb C.
		\end{align}
		Here,  $d_{\mathcal N}=0$ if $(N+2)\theta_0/\pi\notin \N$ and		
		 $q_{{\mathcal N},N+2}$ is a polynomial of order $N+2$ satisfying
		\begin{equation} \label{eq:qNN2-cornscatter2018}
		\Delta q_{{\mathcal N},N+2} = r^N(A \cos N \theta + B \sin N\theta)\quad\mbox{in}\quad \Sigma,\qquad \partial_\nu q_{{\mathcal N},N+2}=0\quad\mbox{on}\quad \partial \Sigma.		\end{equation}

		\noindent {\bf Step 4}: Show that $A = B = 0$.
		
		
	We observe that $r^{N+2}\ln r = o(r^{N+1+\tau})~(r \to 0^+)$ for any $\tau \in (0,1)$ and
		\begin{equation} \label{eq:rDecayRate-cornscatter2018}
		r^{\frac {j\pi} {\theta_0}} \gg r^{N+2} \ln r \gg r^{N+2} \gg r^{N+2+\alpha'}\quad \text{as} \quad r \to 0^+,\qquad\mbox{for all}\quad j\in I(\theta_0, N).
		\end{equation}
		Thus we can conclude from  $\chi u= u_{\mathcal D}^{(1)} + u_{\mathcal D}^{(2)} =  u_{\mathcal N}^{(1)} + u_{\mathcal N}^{(2)}$ in $\Sigma$ that			
		\begin{equation} \label{eq:CnDnj-cornscatter2018}
		C_{\mathcal{D}}=C_{\mathcal N} = 0 \quad \text{and} \quad
		d_{\mathcal D,j} =d_{\mathcal N,j} = 0 \quad \text{for all}~ j\in I(\theta_0, N)
		\end{equation}
		and
		\ben
		q_{{\mathcal D},N+2}(x) + d_{\mathcal D}\, r^{N+2} \sin (N+2)\theta=
		q_{{\mathcal N},N+2}(x) + d_{\mathcal N}\, r^{N+2} \cos (N+2)\theta=:q_{N+2}	\enn
		Combining \eqref{eq:qDN2-cornscatter2018} and \eqref{eq:qNN2-cornscatter2018} yields
		\begin{equation} \label{eq:DeltaqN2-cornscatter2018}
		\Delta q_{N+2} = r^N(A \cos N \theta + B \sin N\theta)
		\end{equation}
		and further
		\begin{equation} \label{eq:Delta2qN2-cornscatter2018}
		\Delta^2 q_{N+2}
		= 0\qquad  \mbox{in } \Sigma,\qquad
		\quad q_{N+2} = \partial_\nu q_{N+2} = 0 \quad \mbox{on } \partial \Sigma.
		\end{equation}
		Applying [Proposition 12, \citen{ElHu2018acoustic}] to \eqref{eq:Delta2qN2-cornscatter2018}, we obtain $q_{N+2} = 0$ in $\Sigma$. Finally, it can be concluded from \eqref{eq:DeltaqN2-cornscatter2018} that $r^N(A \cos N \theta + B \sin N\theta) \equiv 0$ in $\Sigma$, which implies $A = B = 0$. The proof is complete.
	\end{proof}

	In the subsequent two corollaries we present examples of $h$ fulfilling
	the condition (\ref{eq:hform1-cornscatter2018})  in
	 Lemma  \ref{lem:pre-cornscatter2018}.	
	\begin{cor} \label{prop:pre-cornscatter2018}
		Assume in Lemma \ref{lem:pre-cornscatter2018} that  $h\in C^{1,\alpha}(\overline{\Sigma}_\epsilon)$ for some $0<\alpha<1$. Then,		%
	$h(O) = 0$ and $|\nabla h(O)| = 0$.
	\end{cor}
	\begin{proof}
	Since $h\in C^{1,\alpha}(\overline{\Sigma}_\epsilon)$, 	the function $h$ admits the asymptotic behavior
	\begin{equation} \label{eq:hform2-cornscatter2018}
		h(x) = h(O) + \nabla h(O) \cdot x + \mathcal O(|x|^{1+\alpha}), \quad |x| \to 0
	\end{equation}
	in $\Sigma_\epsilon$.
	Note that the expression of $h$ in \eqref{eq:hform2-cornscatter2018} takes the form \eqref{eq:hform1-cornscatter2018} with $N = 0$. Applying Lemma \ref{lem:pre-cornscatter2018} gives $h(O) = 0$. Then we have
		\begin{equation*} 
		h(x) = \nabla h(O) \cdot x + \mathcal O(|x|^{1+\alpha}), \quad |x| \to 0
		\end{equation*}
		which is of the from \eqref{eq:hform1-cornscatter2018} with $N = 1$. Again using Lemma \ref{lem:pre-cornscatter2018}, we arrive at $|\nabla h(O)| = 0$.
	\end{proof}

	\begin{cor} \label{lem:2-cornscatter2018}
	Suppose in Lemma \ref{lem:pre-cornscatter2018} that  $h\in S(A,b)$ where $A=(a_1, a_2)$ and $b$ are analytic near $O$. Then		%
	$h\equiv 0$ in $B_\epsilon (O)$.
	\end{cor}
	\begin{proof}
	We first prove that
	the lowest order non-vanishing term in the Taylor expansion of $h$ at the corner point $O$ is harmonic, if $h$ does not vanish identically.
\rot{This extends the proof of \cite[Proposition 10]{ElHu2018acoustic}
in the special case  $a_1=a_2\equiv 0$ and $b=b_0\in \mathbb C$ to a more general setting.}

Since $A$ and $b$ are both analytic functions, by elliptic regularity theory $h$ is also analytic in $B_\epsilon(O)$.		
By  \cite[Lemma 2.1]{EHY2015}, there exists a positive integer $M\in \N_0$ such that $h$ can be expanded in the polar coordinates $(r,\theta)$ into the convergent series
		\[
		h(x) = \sum_{j \geq M} r^j F_j(\theta), \quad\text{where}\quad F_j(\theta) = \sum_{n+2m = j} (C_{n,m}^+ \cos n\theta + C_{n,m}^- \sin n\theta)
		\]
with $C_{n,m}^\pm \in \mathbb C$. Here, we have $C_{n,m}^\pm=0$ if $n+2m<M$, because $r^M F_M(\theta)$ is supposed to be the lowest order term.	In two dimensions,
		it is easy to  check that
				\begin{align}
		\Delta h
		& = (\frac {\partial^2} {\partial r^2} + \frac 1 r \frac {\partial} {\partial r} + \frac 1 {r^2} \frac {\partial^2} {\partial \theta^2}) h \nonumber\\
		& = \sum_{j \geq M} \left( j(j-1) r^{j-2} F_j + j r^{j-2} F_j + r^{j-2} F_j'' \right)\nonumber\\
		& = \sum_{j \geq M - 2} r^j[ (j+2)^2 F_{j+2} + F_{j+2}'']. \label{eq:lem2.2-1-cornscatter2018}
		\end{align}
		On the other hand, assuming
		\[
		A(x) = \sum_{j \geq 0} r^j A_j(\theta), \qquad b(x) = \sum_{j \geq 0} r^j b_j(\theta),
		\]
		we find
		\begin{equation} \label{eq:lem2.2-2-cornscatter2018}
		A(x) \cdot \nabla h(x) = \sum_{j \geq M - 1} r^j  \tilde{A}_j(\theta)
		\end{equation}
		and
		\begin{align}\nonumber
		b(x)\,h(x)
		 = \big( \sum_{j \geq 0} r^j b_j(\theta) \big) \big( \sum_{\ell \geq M} r^\ell F_\ell(\theta) \big)
		 = \sum_{j \geq M} r^j\; \widetilde b_j(\theta),
		 \end{align}
		  where
		  \be
		   \widetilde b_j(\theta) = \sum_{k + \ell \geq j; k \geq 0, \ell \geq M} b_k(\theta) F_\ell(\theta). \label{eq:lem2.2-3-cornscatter2018}
		\en
		Inserting \eqref{eq:lem2.2-1-cornscatter2018}-\eqref{eq:lem2.2-3-cornscatter2018} into the equation
		\[
		\Delta h(x) + A(x) \cdot \nabla h(x)+ b(x)\, h(x) = 0		\]
		and comparing the coefficients of $r^{M-1}$, we obtain $M^2 F_M + F_M'' = 0, ~ \theta \in [0,2\pi]$. By the definition of $F_j(\theta)$,  it follows that
		\begin{align}
		F_M''(\theta)
		& = (-n^2) \sum_{n + 2m = M} ( C_{n,m}^+ \cos n\theta + C_{n,m}^- \sin n\theta ) \nonumber\\
		& = -M^2\; ( C_{M,0}^+ \cos M\theta + C_{M,0}^- \sin M\theta )
	-n^2\!\!\! \sum_{n + 2m = M, m\neq 0} ( C_{n,m}^+ \cos n\theta + C_{n,m}^- \sin n\theta ) \label{eq:lem2.2-4-cornscatter2018}
		\end{align}
		and
		\begin{align}
		M^2 F_M(\theta)
		& = M^2 ( C_{M,0}^+ \cos M\theta + C_{M,0}^- \sin M\theta ) \nonumber\\
		& \quad + M^2 \sum_{n + 2m = M, m\neq 0} ( C_{n,m}^+ \cos n\theta + C_{n,m}^- \sin n\theta ). \label{eq:lem2.2-5-cornscatter2018}
		\end{align}
Combining the previous two identities we get
		\begin{equation} \label{eq:lem2.2-6-cornscatter2018}
		\sum_{\substack{n+2m=M\\m \neq 0}} ( C_{n,m}^+ \cos n\theta + C_{n,m}^- \sin n\theta ) (M^2 - n^2) = 0\qquad\mbox{for all}\quad \theta \in [0,2\pi].
		\end{equation}
	Since $M^2 - n^2 \neq 0$ for all $(n,m) \in \{ (n,m) \in \mathbb N^2 \,;\, n + 2m = M, m \neq 0 \}$, the relation \eqref{eq:lem2.2-6-cornscatter2018} implies
		\[
		C_{n,m}^\pm = 0 \quad \text{if} \quad n + 2m = M, m \neq 0.
		\]
		Hence,
		\[
		h(x) = r^M \big[ C_{M,0}^+ \cos M\theta + C_{M,0}^- \sin M\theta \big] + \sum_{j \geq M+1} r^j F_j(\theta),
		\]
that is, $h$ is of the form (\ref{eq:hform1-cornscatter2018}) near the corner point $O$.
Applying Lemma \ref{lem:pre-cornscatter2018} we obtain $C_{M,0}^\pm=0$.
This is a contradiction to the fact that $F_M(\theta)$ does not vanish. Hence, $h\equiv 0$ in $B_\epsilon(O)$.
\end{proof}

\section{Determination of convex-polygonal source support
	} \label{sec:contiSource-cornscatter2018}
	
	In this section, we first prove Theorem \ref{thm:Holdercontin-cornscatter2018} for $C^{1,\alpha}$-continuous source functions around a corner point  and then generalize the uniqueness result to a larger class of smooth source functions \rot{with a low contrast to the background medium.}
	
	\begin{proof}[Proof of Theorem \ref{thm:Holdercontin-cornscatter2018}]
		Let $\tilde u$ be the solution of $\Delta \tilde u + k^2 n \tilde u = \tilde f$ in $\R^2$ where $\tilde D := \supp \tilde{f}$ is also a convex polygon and $\tilde{f}\in C^{1, \alpha}(\overline{\tilde{D}\cap B_\epsilon(\tilde{O})})$ near each corner $\tilde{O}$ of $\partial \tilde D$. Suppose further that $|\tilde{f}|+|\nabla \tilde{f}|>0$ at each corner point of  $\partial \tilde{D}$.
		Assume that
		\begin{equation} \label{eq:uEqutilde-cornscatter2018}
			u^\infty(\hat x) = \tilde u^\infty(\hat x) \quad\text{for all}~\hat x \in \mathbb{S}.
		\end{equation}
		Applying Rellich's lemma and the unique continuation for Helmholtz equations, we see \begin{equation}\label{eq:trace}
			u=\tilde{u} \quad\mbox{in}\quad \R^2\backslash\overline{D\cup \tilde{D}}.
		\end{equation}
		Note that here we have used the assumption that the refractive index function $n(x)$ is given. 	
		
		The rest of the proof  of Theorem  \ref{thm:Holdercontin-cornscatter2018}  is divided into two steps.  In the first step we prove the unique determination of the source support, and in the second step the determination of \rot{the zeroth and first order derivatives of the source term} at corner points.

		\noindent {\bf Step 1}: Prove $D=\tilde{D}$.  If the geometric shapes of $D$ and $\tilde D$ are not identical, without loss of generality we may suppose there exists some corner $O\in \partial D$ and a neighborhood $B_\epsilon(O)$ of $O$ such that $B_\epsilon(O) \cap \tilde D = \emptyset$.
		Set $\Sigma_\epsilon := B_\epsilon(O) \cap D$ with the opening angle $\theta_0\in (0, \pi)$, and write $\Gamma_\epsilon:= \partial D \cap B_\epsilon(O)$.
		This scenario is illustrated in Figure \ref{fig:1-cornscatter2018}.
		
		
		\begin{figure}[htbp]
			\centering
			\includegraphics[width = 3.8in]{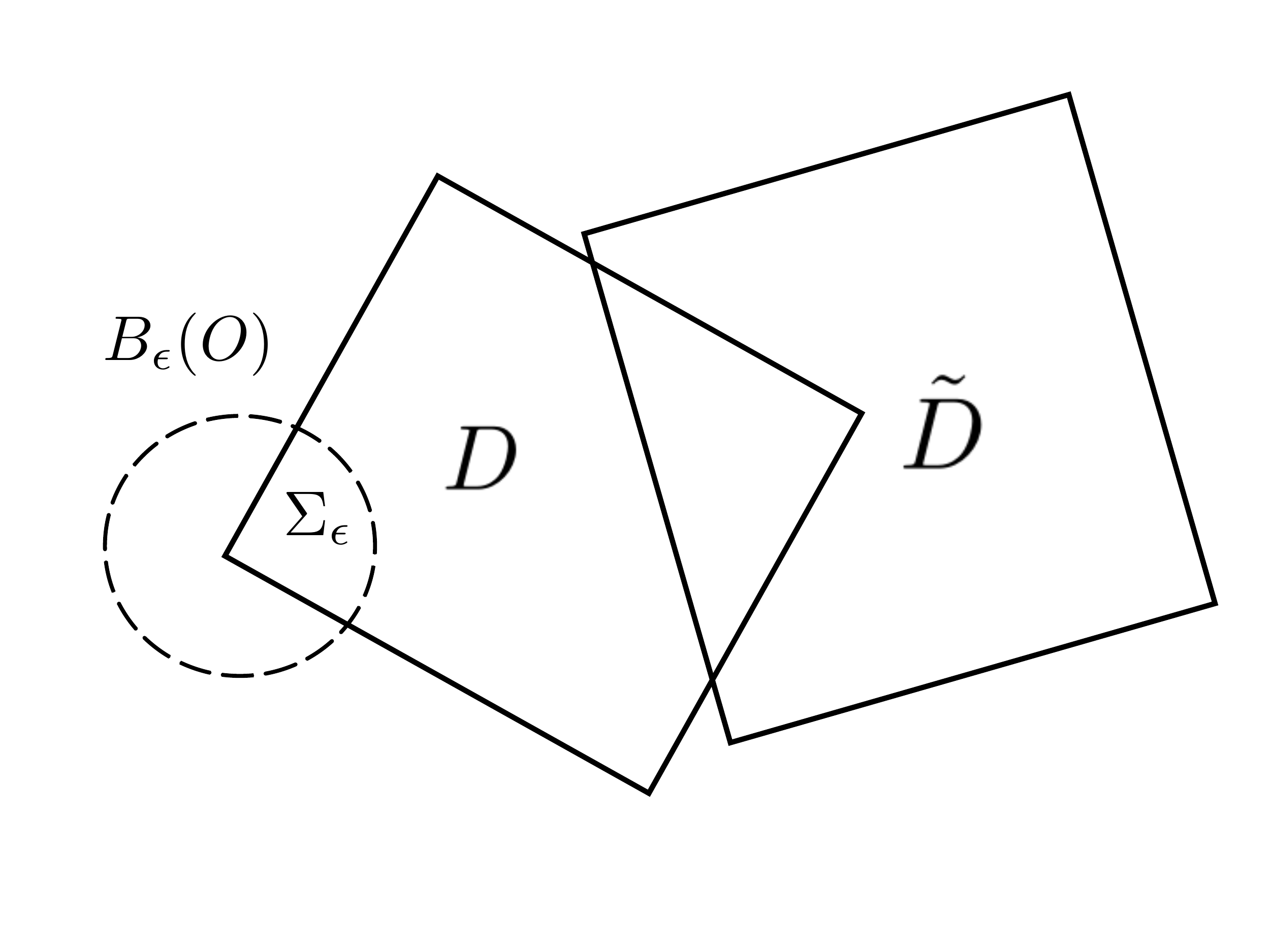}
			\caption{Illustration of two convex-polygonal source supports $D$ and $\tilde{D}$. }
			\label{fig:1-cornscatter2018}
		\end{figure}		
		By coordinate translation, we may suppose
		without loss of generality that the corner point $O$ is located at the origin. Then we have
		\begin{equation} \label{eq:uutilde-cornscatter2018}
			\begin{cases}
				\Delta u + k^2 n(x) u = f & \mbox{in } \Sigma_\epsilon, \\
				\Delta \tilde u + k^2 n(x) \tilde u = 0 & \mbox{in } \Sigma_\epsilon.
			\end{cases}
		\end{equation}
		Since $f \in L^2(B_\epsilon (O))\cap C^{1,\alpha}(\overline{D\cap \rot{B_\epsilon(O)}})$,  by standard elliptic regularity theory we have  $u, \tilde u \in H^2(B_\epsilon(O))$.
		Recall the transmission conditions for $u$:	
		\ben
		u^+=u^-,\qquad \partial_\nu u^+=\partial_\nu u^-\qquad\mbox{on}\quad \Gamma_\epsilon,
		\enn	
		where $(\cdot)^\pm$ denote the traces of $u\in H^2(B_\epsilon(O))$ taking from $D\,  (+)$ and $\R^2\backslash \overline{D} \, (-)$, respectively.
		We deduce from (\ref{eq:trace}) and the previous transmission conditions that
		\begin{equation} \label{eq:uutildeGamma-cornscatter2018}
			u=\tilde{u}, \quad \partial_\nu u=\partial_\nu \tilde{u}\qquad \mbox{on}\quad \Gamma_\epsilon.
		\end{equation}
		%
		Set $w = u - \tilde u$.
		From \eqref{eq:uutilde-cornscatter2018} and \eqref{eq:uutildeGamma-cornscatter2018} it follows that $w\in H^2(B_\epsilon(O))$ solves the Cauchy problem
		\begin{equation} \label{eq:wDNBC-cornscatter2018}
			\begin{cases}
				\Delta w + k^2 n(x) w = f_0 + f_1 & \mbox{in } \Sigma_\epsilon, \\
				w = \partial_\nu w = 0 & \mbox{on } \Gamma_\epsilon,
			\end{cases}
		\end{equation}
		where
		\[
		f_0(x) := \begin{cases}
		f(O), &\mbox{ when } f(O) \neq 0, \\
		r [(\partial_1 f)(O) \cos\theta + (\partial_2 f)(O) \sin\theta], &\mbox{ when } f(O) = 0,
		\end{cases}
		\]
		and
		\[
		f_1(x) := f(x) - f_0(x),\quad x\in \Sigma_\epsilon.
		\]
		The assumption $|f(O)| + |\nabla f(O)| > 0$
		gives $f_0 \not\equiv 0$.
		However, noting that $f\in C^{1, \alpha}(\overline{B_\epsilon(O)\cap D})$ and $f_0$ is harmonic, we deduce from  Lemma \ref{lem:pre-cornscatter2018} that $f_0 \equiv 0$. This contradiction implies $\partial D = \partial \tilde D$.

		{\bf Step 2}: Prove $f(O)=\tilde{f}(O)$, $\nabla f(O)=\nabla \tilde{f}(O)$ where $O$ is an arbitrary corner point of the source support.
		
		We still use the notations $\Sigma_\epsilon$ and $\Gamma_\epsilon$ defined in Step 1 with $D=\tilde{D}$.  Repeating the arguments in Step 1,  we obtain
		\begin{equation*}
			\begin{cases}
				\Delta u + k^2 n(x) u = f & \mbox{in } \Sigma_\epsilon, \\
				\Delta \tilde u + k^2 n(x) \tilde u = \tilde f & \mbox{in } \Sigma_\epsilon, \\
				u = \tilde u,\, \partial_\nu u = \partial_\nu \tilde u & \mbox{on } \Gamma_\epsilon.
			\end{cases}
		\end{equation*}
		This implies that $w := u - \tilde u \in H^2(B_\epsilon(O))$ is a weak solution of
		\begin{equation*}
			\begin{cases}
				\Delta w + k^2 n(x) w = f - \tilde f & \mbox{in } \Sigma_\epsilon, \\
				w = \partial_\nu w = 0 & \mbox{on } \Gamma_\epsilon.
			\end{cases}
		\end{equation*}
		Analogous to Step 1, we may define
		\[
		f_0(x) := \begin{cases}
		f(O) - \tilde f(O), &\mbox{ when } f(O) \neq \tilde f(O) \\
		r \{ [(\partial_1 f)(O) - (\partial_1 \tilde f)(O)] \cos\theta + [(\partial_2 f)(O) - (\partial_2 \tilde f)(O)] \sin\theta \}, &\mbox{ when } f(O) = \tilde f(O),
		\end{cases}
		\]
		and
		\[
		f_1(x) = f(x) - \tilde f(x) - f_0(x).
		\]
		Applying Lemma \ref{lem:pre-cornscatter2018} again, we obtain $f_0(x) \equiv 0$, which implies $f(O) =\tilde f(O)$ and $\nabla f(O) =\nabla \tilde f(O)$. Since the corner $O$ is taken arbitrarily, we finish the proof of Theorem \ref{thm:Holdercontin-cornscatter2018}.	\end{proof}
	If the smoothness of $f$ at the corner points can be weakened to be $f \in C^{0,\alpha}(\overline{D \cap B_\epsilon(O)})$, it follows from the proof of
	Theorem \ref{thm:Holdercontin-cornscatter2018} that the source support  $\partial D$ and $f(O)$  can be uniquely determined by $u^\infty(\hat{x})$ for all $\hat{x}\in\mathbb{S}$, if $f(O)\neq 0$. This proves Corollary \ref{cor1}.  Combining the arguments  in proving
	Theorem \ref{thm:Holdercontin-cornscatter2018} and Lemma \ref{lem:pre-cornscatter2018}, we immediately get the shape identification result of Corollary \ref{cor3} (i).
	
	Theorem \ref{thm:Holdercontin-cornscatter2018} can be generalized to more smooth source terms with a compact support, \rot{under an extra condition concerning the source discontinuity at corner points (see (\ref{eq:assumption1}) below)}. This leads to \rot{Corollary} \ref{prop:smooth-cornscatter2018}.
	\begin{cor} \label{prop:smooth-cornscatter2018}
		Suppose that $D := \supp (f) \subset \R^2$ is a convex polygon and $n(x)\equiv 1$ in $\R^2$.  Assume that there exists an $l\in \N_0$ such that
		$$f \in C^{l+1,\alpha}(\overline{D \cap B_\epsilon(O)}) \cap W^{l,\infty}(B_\epsilon(O)),\qquad 0 < \alpha <1$$
		for each corner $O$ of $\partial D$, and that there exist a multi-index $\beta=(\beta_1, \beta_2)$  ($\beta_j\in \N_0$) with $|\beta| := \beta_1+\beta_2=l$ such that
		\begin{equation}\label{eq:assumption1}
			| \partial^\beta f(O)|+|\partial^{\beta'} f(O)|> 0,\quad \beta' \geq \beta,\, |\beta' - \beta| = 1.
		\end{equation}
		Then $\partial D$, $\partial^\beta f(O)$ and $\partial^{\beta'} f(O)$ can be uniquely determined by $u^\infty(\hat{x})$ for all $\hat{x}\in\mathbb S$.
	\end{cor}
	\begin{rem} \label{rem3.1}
		(i)	By Sobolev embedding theorems, the condition $f\in W^{l,\infty}(B_\epsilon(O))	$ implies that $f\in C^{l-1, \alpha}(B_\epsilon(O))$. Hence, $\nabla^{j}f(O)=0$ for all $|j|\leq l-1$. \rot{
If $l = 0$, Corollary \ref{prop:smooth-cornscatter2018} is equivalent to the result of Theorem \ref{thm:Holdercontin-cornscatter2018}
when
$n(x)\equiv 1$. }
		(ii) The multi-index $\beta' $ in (\ref{eq:assumption1}) takes the form $\beta'=(\beta_1+1, \beta_2)$ or $\beta'=(\beta_1, \beta_2+1)$ \rot{in two dimensions}.
	\end{rem}
	
	\begin{proof}[Proof of \rot{Corollary} \ref{prop:smooth-cornscatter2018}]
		Obviously, we have $f \in H^l(B_\epsilon(O))$, and by the regularity of elliptic equations (see e.g., \cite{GT}) we get ${u \in H^{l+2}(B_\epsilon(O))}$. By the trace lemma,
		\[
		\partial_\nu^{j}u^+=\partial_\nu^j u^-,\qquad j=0,1,\cdots, l+1.
		\]
		Proceeding as in the proof of Theorem \ref{thm:Holdercontin-cornscatter2018}, we suppose there exist two sources $f$ and $\tilde f$ ($\tilde D :=\supp \tilde f$ is a convex polygon) which generate identical far-field patterns over all observation directions.	If $\partial D \neq \partial \tilde{D}$, we suppose there exists a corner point $O$ of $D$ and a neighborhood $B_\epsilon(O)$ of $O$ such that $B_\epsilon(O) \cap \tilde D = \emptyset$.
		Setting	 $\Sigma_\epsilon = D \cap B_\epsilon(O)$ and $\Gamma_\epsilon = \partial D\cap B_\epsilon(O) $.  In this section we define $v = \partial^{\beta} (u - \tilde u)$, where $\beta=(\beta_1, \beta_2)$, $|\beta|=l$, is the multi-index specified in Corollary \ref{prop:smooth-cornscatter2018}. Then we have
		\begin{equation} \label{eq:uutildeho-cornscatter2018}
			\begin{cases}
				\Delta v + k^2  v = \partial^\beta f & \mbox{in}\; \Sigma_\epsilon, \\
				v = \partial_\nu v = 0 & \mbox{on } \Gamma_\epsilon.
			\end{cases}
		\end{equation}
		Applying Lemma \ref{lem:pre-cornscatter2018} to \eqref{eq:uutildeho-cornscatter2018}, we conclude
		$ \partial^\beta f(O)=\partial^{\beta'}f(O)= 0$ for all $\beta'\geq \beta$, $|\beta'-\beta|=1$, which contradicts  our assumption \eqref{eq:assumption1}. Thus $D = \tilde D$. In the same manner, one can prove
		\[
		\partial^\beta (f-\tilde{f})(O)=\partial^{\beta'} (f-\tilde{f})(O)=0.
		\]


		The proof is complete.
	\end{proof}

	\section{Determination of source terms} \label{sec:harmoSource-cornscatter2018}
	
	In this section we prove Theorem \ref{thm:analytic-cornscatter2018}.
	\begin{proof}[Proof of Theorem \ref{thm:analytic-cornscatter2018}]
		Assume that the two convex polygons $D$ and $\tilde{D}:=\text{supp}(\tilde{f})$ produce identical far-field patterns. Here we suppose that $\tilde{f}=\tilde{v}|_{\overline{\tilde{D}}}$ for some $\tilde{v}\in S(A, b)$ which is analytic near corner points of $\tilde{D}$. Denote by $\tilde{u}$ the Sommerfeld radiation solution corresponding to $\tilde{D}$ and $\tilde{f}$.
		
		If $D\neq\tilde{D}$, we may choose \rot{at} least one corner point		
		$O\in\partial D$ as done in the proof of Theorem \ref{thm:Holdercontin-cornscatter2018}. Since $f=v|_{\overline{D}}$ for some $v\in S(A, b)$, the source function $f$ must be analytic on $\overline{B_\epsilon(O)\cap D}$.	 By \rot{the proof of} Corollary \ref{lem:2-cornscatter2018}, the lowest order term in the Taylor expansion of $f$ in $B_\epsilon(O)\cap D$ is harmonic.
		Arguing analogously to the proofs of Corollary \ref{lem:2-cornscatter2018} and
		Theorem \ref{thm:Holdercontin-cornscatter2018}, we may conclude from $u^\infty=\tilde{u}^\infty$ that
		$f$ vanishes identically on $\overline{B_\epsilon(O)\cap D}$.
		Hence $v\equiv 0$  on $\overline{B_\epsilon(O)\cap D}$, and
		by unique continuation we get $v\equiv 0$ on $\overline{D}$.
		This implies that  $f$ vanishes identically on $\overline{D}$ and thus $\text{supp}(f)=\emptyset$,
		which contradicts our assumption.  Hence, we obtain the uniqueness in determining the source support.
		
		To determine the source term, we set $w=u-\tilde{u}$ and consider the Cauchy problem (cf. \eqref{eq:wDNBC-cornscatter2018})
		\begin{equation*}
			\begin{cases}
				\Delta w + k^2 n(x) w =  f - \tilde f & \mbox{in } \Sigma_\epsilon, \\
				w = \partial_\nu w = 0 & \mbox{on } \Gamma_\epsilon,
			\end{cases}
		\end{equation*}
		where $\Sigma_\epsilon$ and $\Gamma_\epsilon$ are defined as the same ones in the proof of Theorem \ref{thm:Holdercontin-cornscatter2018}.
		Note that $f - \tilde f=(v-\tilde v)|_{\overline{D}}$ where $v-\tilde v\in S(A, b)$ is analytic in $B_\epsilon(O)$. By the proof of Lemma \ref{lem:2-cornscatter2018}, the function $f - \tilde f$ takes the form
		\[
		f(x) - \tilde f(x) = r^N(A \cos N \theta + B \sin N\theta) + \mathcal O(r^{N+\alpha}), \quad |x| \to 0,\, x \in \Sigma_\epsilon.
		\]
		for some $N\in \N_0$.	Recalling Lemma \ref{lem:pre-cornscatter2018} and  Corollary \ref{lem:2-cornscatter2018},	 we arrive at  $f=\tilde f$ and thus $v=\tilde v$ on $\overline{\Sigma}_\epsilon$. Since $A$ and $b$ are a priori given, the relation $f=\tilde{f}$ in $D$ follows from the unique continuation property of elliptic equations.
	\end{proof}

	\section{Characterization of radiating sources and singularity at corner points} \label{sec:radiatS-cornscatter2018}
	
	\begin{proof}[Proof of Corollary \ref{thm:NonRadiatingSource-cornscatter2018}]
		
		Assume $u^\infty \equiv 0$. Then we obtain $u \equiv 0$ in $\R^3 \backslash \overline D$ by Rellich's lemma, and in particular the traces of $u$ vanish on $\Gamma_\epsilon=\partial D\cap B_\epsilon(O)$ for some $\epsilon>0$.
		Under the condition (i) in Corollary \ref{thm:NonRadiatingSource-cornscatter2018},
		we can deduce from the proof of Corollary \ref{prop:smooth-cornscatter2018} that $v=\partial^\beta u$ satisfies	
		\begin{equation} \label{eq:utildeNonRa1-cornscatter2018}
			\begin{cases}
				\Delta v + k^2n_0 v = \partial^\beta f & \mbox{in } \;\Sigma_\epsilon=D\cap B_\epsilon(O), \\
				v = \partial_\nu v = 0 & \mbox{on}\; \Gamma_\epsilon.
			\end{cases}
		\end{equation} It then follows that
		$\partial^\beta f(O)=\partial^{\beta'}f(O)= 0$ for the indexes $\beta$ and  $\beta'$ specified in  Corollary \ref{prop:smooth-cornscatter2018} (i), which contradicts the condition \eqref{eq:assumption}.
		
		If $f$ fulfills the condition (ii), by the proof of Theorem \ref{thm:analytic-cornscatter2018}, $f$ must vanish identically, leading to a contradiction to the assumption that $D=\text{supp}(f) \neq \emptyset$.
	\end{proof}
	
	\begin{proof}[Proof of Corollary \ref{cor4}] Since $n(x)\equiv n_0$ in $B_R\backslash\overline{D}$, $u$ is analytic in $B_R\backslash\overline{D}$. Suppose on the contrary that $u$ can be analytically extended from $(B_R\backslash\overline{D})\cap B_\epsilon(O)$  to $D\cap B_\epsilon(O)$ for some $\epsilon>0$.  Denote by $w$ the extended solution in $B_\epsilon(O)$, which solves the Helmholtz equation $\Delta w+k^2n_0 w=0$ in $B_\epsilon(O)$ and coincides with $u$ in
		$(B_R\backslash\overline{D})\cap B_\epsilon(O)$. Setting $v=u-w$, we may arrive at  the same boundary value  problem (\ref{eq:utildeNonRa1-cornscatter2018}) with $l=0$ and then the same contradictions as in the proof of Corollary \ref{thm:NonRadiatingSource-cornscatter2018}.
	\end{proof}	
	
	\begin{proof}[Proof of Corollary \ref{cor3} (ii)] If the lowest order expansion of $f$ around the corner is harmonic, the results of
		Corollaries \ref{thm:NonRadiatingSource-cornscatter2018}  and \ref{cor4} can be proved analogously by
		applying Corollary \ref{lem:2-cornscatter2018}.
	\end{proof}
	
\section{A data-driven inversion scheme}\label{inversion}
The aim of this section is to propose a non-iterative numerical scheme to reconstruct the shape $\partial D$ of a convex-polygonal source support embedded in $B_R$ from a single far-field pattern. For simplicity we shall assume $n(x)\equiv n_0>0$ in $|x|<R$ and consider a slightly different model as follows
\begin{eqnarray}\label{Equation}\left\{\begin{array}{lll}
\Delta u+k^2u=0&&\mbox{in}\quad |x|>R,\\
\Delta u+k^2n_0 u=f&&\mbox{in}\quad |x|<R,\\
u^+=u^-,\; \partial_\nu u^+ =\lambda\, \partial_\nu^- u&&\mbox{on}\quad |x|=R,
\end{array}
\right.
\end{eqnarray}
where $u$ is required to fulfill the Sommerfeld radiation condition at infinity.
Here, the medium inside $B_R$ is supposed to be homogeneous and isotropic, and the parameter $\lambda>0$ ($\lambda\neq 1$) models the ratio of the constant densities in $|x|>R$ and $|x|<R$. 
 In this section, the source function $f$ is supposed to satisfy the two conditions in Theorem \ref{thm:Holdercontin-cornscatter2018}.
Although the refractive index $n(x)$ is discontinuous on $\partial B_R$, our corner scattering theory explored in previous sections are still valid for the new model (\ref{Equation}).
By Theorem \ref{thm:Holdercontin-cornscatter2018} and Corollary \ref{cor1}, the source support  $\partial D$ can be uniquely determined by the radiated far-field pattern $u^\infty(\hat{x})$ for all $\hat{x}\in\mathbb{S}$. Moreover, by the proof of Corollary \ref{cor4} we obtain 
\begin{prop}\label{P1}
The solution $u$ cannot be (analytically) extended from $B_R\backslash\overline{D}$ into a neighboring area of any corner point of $\partial D$ as a solution of the Helmholtz equation $\Delta u+k^2n_0 u=0$.
\end{prop}

To determine the position and shape of $\partial D$, we introduce a test domain $\Omega$ which represents an acoustically sound-soft obstacle inside $B_R$. Here $\Omega \subset B_R$ is called a test domain if it is a connected convex domain and $k^2n_0$ is not the Dirichlet eigenvalue of $-\Delta$ over $\Omega$.
Consider the time-harmonic scattering of a plane wave from the obstacle $\Omega$ modeled by
\ben
\Delta v+k^2 v=0\quad&&\mbox{in}\quad \R^2\backslash\overline{B}_R, \quad v=e^{ikx\cdot d}+v^{sc},\\
\Delta v+k^2n_0 v=0\quad&&\mbox{in}\quad B_R\backslash \overline{\Omega},\\
v=0\quad&&\mbox{on}\quad \partial \Omega, \\
v^+=v^-,\; \partial_\nu v^+ =\lambda\, \partial_\nu^- v&&\mbox{on}\quad \partial B_R,\enn
where $v^{sc}$ is a Sommerfeld radiating solution in $|x|>R$. Note that $d\in \mathbb{S}$ denotes the incident direction of the plane wave $v^{in}(x)=e^{ikx\cdot d} $ in $\R^2$. The far-field patterns $v^{\infty}(\hat{x},d)$ of $v^{sc}$ for all incident directions define the far-field operator $F^{(\Omega)}: L^2(\mathbb{S})\rightarrow L^2(\mathbb{S})$:
\ben
(F^{(\Omega)}g)(\hat{x})=\int_{\mathbb{S}} v^\infty(\hat{x},d) g(d)\,ds(d),\quad \hat{x}\in \mathbb{S}.
\enn
Analogously, denote by $F_0: L^2(\mathbb{S})\rightarrow L^2(\mathbb{S})$
the far-field operator corresponding to the penetrable scatterer $B_R$ without the embedded sound-soft obstacle $\Omega$, and by
$\mathcal{S}_0$ the scattering operator corresponding to $F_0$ defined by
\ben
\mathcal{S}_0:=I+2ik\gamma F_0,\quad \gamma:=\frac{e^{i\pi/4}}{\sqrt{8k\pi}}.
\enn
It was proved in \cite{BKL} via Factorization Method that the spectrum system $(\lambda_j^{(\Omega)}, \psi_j^{(\Omega)})$ of the operator
\ben
F^{(\Omega)}_\#:=|\mbox{Re} ((F_0-F^{(\Omega)})\mathcal{S}_0)|+|\mbox{Im} ((F_0-F^{(\Omega)})\mathcal{S}_0)|
\enn
can be used to characterize the embedded obstacle $\Omega$. Note that the spectra are uniquely determined by $\Omega, B_R, k, n_0$ and $\lambda$, all of which are known. To proceed, we
recall that the data-to-pattern operator $G^{(\Omega)}: H^{1/2}(\partial \Omega)\rightarrow L^2(\mathbb{S})$ is defined as $G^{(\Omega)}(h):=v^{\infty}$, where $v^{\infty}$ is the far-field pattern of the unique radiating solution $v$ to the boundary value problem
\ben
&&\Delta v+k^2 v=0\quad\mbox{in}\quad \R^2\backslash\overline{B}_R, \qquad \Delta v+k^2n_0 v=0
\quad\mbox{in}\quad B_R\backslash \overline{\Omega},\\
&&v=h\quad\mbox{on}\quad \partial \Omega, \qquad\qquad\qquad\quad
v^+=v^-,\; \partial_\nu v^+ =\lambda\, \partial_\nu^- v\quad\mbox{on}\quad \partial B_R.\enn
Our imaging scheme is essentially based on the following lemma.
\begin{lem}\label{Lem6.1}
Let $u^\infty$ be the far-field pattern corresponding to the source term $f$ with the convex polygonal support $D\subset B_R$.   Then $u^\infty$ belongs to the range of the data-to-pattern operator $G^{(\Omega)}$ if and only if $D\subset \overline{\Omega}$.
\end{lem}
\begin{proof}
Suppose $D\subset\overline{\Omega}\subset B_R$. By the definition of $G^{(\Omega)}$, we have $G^{(\Omega)}(h)=u^\infty$ with $h:=u|_{\partial \Omega}$, implying that $u^\infty$ is indeed in the range of $G^{(\Omega)}$. Now, suppose that $u^\infty=G^{(\Omega)}(h)$ for some $h\in H^{1/2}(\Omega)$ but $D\subset \overline{\Omega}$ does not hold.
Since $D$ is a convex polygon and $\Omega$ is a connected convex domain, one can always find a corner point of $\partial D$ which lies in the exterior of $\Omega$.
Using Rellich's lemma and the unique continuation for solutions of the Helmholtz equation,  $u$ can be analytically extended into a full neighborhood of this corner point. However, this is a contradiction to Proposition \ref{P1}, implying that $D\subset \overline{\Omega}$.
\end{proof}
By \cite{BKL}, the range of $G^{(\Omega)}$ coincides with the range of $(F_\#^{(\Omega)})^{1/2}$. Hence, $u^\infty$ belongs to the range of $(F_\#^{(\Omega)})^{1/2}$ if and only if $D\subset \overline{\Omega}$. This reveals an inclusion relation between the unknown target $D$ and 
the a priori given obstacle $\Omega$ through their measurement data $u^\infty$ and $F^{(\Omega)}$. 
Together with Picard's criterion, we derive a non-iterative inversion scheme stated as below.
\begin{thm}\label{Thm}
Introduce the domain-defined indicator
\be\label{Indicator}
W(\Omega):=\sum_{j=1}^\infty\frac{|<u^\infty, \psi^{(\Omega)}_j>_{L^2(\mathbb{S})}|^2}{|\lambda^{(\Omega)}_j|}.
\en
We have $W(\Omega)<\infty$ if and only if $D\subset \overline{\Omega}$. Hence, the source support $D$ can be regarded as the intersection of all test domains $\Omega$ such that $W(\Omega)<\infty$.
\end{thm}
By Theorem \ref{Thm}, the source support $D$ can be reconstructed by firstly selecting different test domains $\Omega\subset B_R$ and then computing the indicator $W(\Omega)$ to get the inclusion relation between $D$ and $\Omega$. In particular, these test domains can be taken as sound-soft disks with different centers and radii. We refer to \cite{MH} for numerical examples of recovering convex polygonal sound-soft obstacles and source support in a homogeneous background medium.  If $\partial D$ contains no corners, only partial information of $D$ can be recovered; see \cite{LS} where the linear sampling method with a single far-field pattern was discussed. In the case of a variable inhomogeneous background medium, Theorem \ref{Thm} can be similarly justified by taking inspirations from the corresponding factorization schemes proved in \cite{GMMR, NPT}.

In Theorem \ref{Thm} we establish a domain-defined factorization method in an inhomogeneous background medium using only a single far-field pattern. A one-wave factorization method was discussed earlier in  \cite{EH2019} for inverse elastic scattering problems. The  pointwise-defined factorization scheme \cite{BKL} makes use of far-field data over all incident directions at a fixed energy. In comparison \cite{BKL}, we use in Lemma \ref{Lem6.1}  a sampling domain $\Omega\subset\R^2$ in place of a sampling point $z\in \R^2$. Correspondingly, the far-field data of the background Green's tensor $G_0(\cdot, z)$ (or its equivalence function used in \cite{BKL}) has been
substituted by the measurement $u^\infty$ in Lemma \ref{Lem6.1}, and the role of the singularity of $G_0$ was replaced by
the absence of analytical extension in a corner domain (see Proposition \ref{P1}) (cf. \cite{BKL} and Theorem \ref{Lem6.1}).
We refer to \cite{KPS} and \cite[Chapter 15]{NP2013} for other domain-defined sampling methods with a single measurement data.

One can interpret Theorem \ref{Thm} as
 a data-driven reconstruction scheme, because far-field patterns of a priori given obstacles (test domains) are involved in inversion process. Theorem \ref{Thm} explains how do these a priori data encode the information of our target. The test domains $\Omega$ can also be replaced by sound-hard, impedance and penetrable scatterers embedded in $B_R$.


	\section*{Acknowledgements} \label{sec:acknow-cornscatter2018}
	The work of J.~Li was partially supported by the NSF of China under the grant No.~11571161 and 11731006, the Shenzhen Sci-Tech Fund No.~JCYJ20170818153840322.

\end{document}